\documentclass[10pt]{amsart}

\usepackage{latexsym}
\usepackage{amssymb}
\usepackage{amsmath}
\usepackage{xcolor}
\usepackage{tikz-cd}
\usepackage{mathrsfs}

\usepackage[italian, english]{babel}
\usepackage{url}

\newtheorem{theorem}{Theorem}[section]
\newtheorem{lemma}[theorem]{Lemma}
\newtheorem{proposition}[theorem]{Proposition}
\newtheorem{corollary}[theorem]{Corollary}

\theoremstyle{definition}
\newtheorem{definition}[theorem]{Definition}

\theoremstyle{remark}
\newtheorem{remark}[theorem]{Remark}

\def\N{\mathbb{N}}

\def\R{\mathbb{R}}

\def\s{{\mathfrak{s}}}

\def\w{{\mathfrak{w}}}

\def\hN{{{}^*\N}}

\def\hA{{{}^*\!A}}
\def\hB{{{}^*\!B}}

\def\hA{{{}^*A}}
\def\hB{{{}^*B}}
\def\hX{{{}^*X}}

\def\hX{{{}^*X}}

\def\F{\mathcal{F}}

\def\U{\mathcal{U}}
\def\V{\mathcal{V}}
\def\W{\mathcal{W}}
\def\X{\mathcal{X}}

\def\UU{\mathfrak{U}}

\def\HJ{{\mathbb{HJ}}}

\def\ueq{{\,{\sim}_{{}_{\!\!\!\!\! u}}\;}}
\def\nueq{{\,{\not\sim}_{{}_{\!\!\!\! u}}\,}}

\makeatletter
\newcommand{\ostar}{\mathbin{\mathpalette\make@circled\star}}
\newcommand{\make@circled}[2]{%
  \ooalign{$\m@th#1\smallbigcirc{#1}$\cr\hidewidth$\m@th#1#2$\hidewidth\cr}%
}
\newcommand{\smallbigcirc}[1]{%
  \vcenter{\hbox{\scalebox{0.77778}{$\m@th#1\bigcirc$}}}%
}
\makeatother

\begin{document}

\title[An extension of Hales-Jewett Theorem]
{A simultaneous extension of Ramsey, Hindman, and
Hales-Jewett Theorems}

\author{Mauro Di Nasso}
\author{Renling Jin}



\begin{abstract}
We prove a multidimensional extension of a strong
Hales-Jewett theorem that
simultaneously and ``directly" extends Ramsey's theorem and Hindman's theorem.
The proofs show the effectiveness and simplicity of the techniques
based on iterated nonstandard extensions that have been recently developed.
Unlike existing ultrafilter proofs, our arguments to prove the strong
Hales-Jewett theorem assume neither minimal
nor idempotent ultrafilters. To demonstrate this, we translate our proof of the
strong Hales-Jewett theorem into an ultrafilter proof that requires
only non-principal ultrafilters.
\end{abstract}

\maketitle

\section*{Introduction}

Van der Waerden's famous theorem \cite{vdw} of 1927 states that in
any finite coloring (partition) $\N=C_1\cup\ldots\cup C_r$ of the natural numbers,
one can find arbitrarily long arithmetic progressions that are ``monochromatic,"
that is, included in a color $C_j$.
Hales-Jewett theorem \cite{hj}
of 1963 is a strong Ramsey property of the
set of words $W_A$ over a finite alphabet $A$, from which
Van der Waerden's theorem can be directly derived as a particular case.
Another now classic result of arithmetic Ramsey theory is Hindman's theorem \cite{hi} of 1974,
which states that for every finite coloring of $\N$ there exists an increasing sequence
$(x_n)$ such that all finite sums of distinct elements $x_{n_1}+\ldots+x_{n_k}$ are monochromatic.

There has been intense research aimed at simultaneously
extending two or three of the Ramsey, Hindman, and Van der Waerden's theorems.
Among the most relevant results in this research are Milliken-Taylor theorem \cite{milliken,taylor}
and its polynomial generalizations \cite{bhw} that
simultaneously extend Ramsey and Hindman's theorem;
some of the Bergelson-Hindman results in \cite{bh}
that simultaneously extend Ramsey and Van der Waerden's theorems
(see also Samet-Tsaban paper \cite{st});
and Beiglb\"{o}ck's theorem \cite{beiglbock},
that simultaneously extends Ramsey's theorem and Central Set theorem.
Although Central Set theorem implies both Hindman's theorem and Van der Waerden's theorem,
these implications are not ``simultaneous''
(in the style that Milliken-Taylor theorem simultaneously implies
Ramsey and Hindman's theorems) in \cite{bh,beiglbock}.

As a further step, it is natural to search for similar simultaneous extensions
where Van der Waerden's theorem is replaced by Hales-Jewett theorem or even
by the strong Hales-Jewett theorem (Theorem \ref{HJstrong}).
About this, Bergelson and Hindman \cite{bh2} proved that a non-commutative
version of the Central Set theorem
implies both Hales-Jewett and Hindman's theorem although these implications
are again not simultaneous (see also \cite[Theorem 4.10 and Theorem 5.1]{hindman}).
We have not been able to find in the literature any result that simultaneously
extends all three theorems, with the sole exception of the Carlson-Simpson
 theorems \cite{cs,carlson}.
We emphasize, however, that this latter generalizations are not ``direct" at all,
and in fact all the works \cite{bh,bh2,bbh,beiglbock,bhw,st} mentioned above
appeared subsequently with the aim of more direct extensions of the
fundamental theorems considered.
(Here we say that Theorem A is a ``direct" extension
of Theorem B to mean that the latter is literally a special case of the formulation of the former.)

In this paper, using iterated hyper-extensions, we prove
the following ``direct" simultaneous multidimensional extension
of Ramsey, Hindman, and the strong Hales-Jewett theorems.\footnote
{~This is a (slightly) simplified version of Theorem \ref{RHHJ}.
See Remarks \ref{maintheoremimpliesHJ} and \ref{maintheoremimpliesMT}
to see how this theorem has Hales-Jewett, Ramsey, and Hindman's theorems
as particular cases.}
\begin{itemize}
\item
\emph{Let $W=W_A$ be the set of words over an alphabet $A$,
let $B\subset W$ be a finite set, and let $k\in\N$.
For every finite coloring of the $k$-tuples of words $[W]^k$ there exists a sequence of variable words
$(w_n(v))_{n\in\N}$ of increasing lengths such that
for every sequence $(b_n)_{n\in\N}$ in $B$
and for all nonempty finite subsets $F_1<F_2<\cdots<F_k$ of $\N$,
the following $k$-tuples are monochromatic:}\footnote
{~If $F=\{n_1<\ldots<n_s\}$, we denote $\prod_{i\in F}w_i(b_i)$ the concatenation of
the finite list of words $w_{n_1}(b_{n_1}),\ldots,w_{n_s}(b_{n_s})$.}
$$\left\{\prod_{i\in F_1}w_i(b_i),
\prod_{i\in F_2}w_i(b_i),\,\ldots\,,
\prod_{i\in F_k}w_i(b_i)\right\}.$$
\end{itemize}

The use of iterated hyper-extensions is a recent nonstandard analysis technique
that has revealed useful for proving new results in Ramsey theory
(see \cite{dj} and \cite[Chapter 4]{dgl}, and the references therein).
In essence, as it will be made precise later, iterated hyper-extensions
correspond to the iterated tensor products of ultrafilters,
with the advantage of a simpler formalism that facilitate combinatorial
constructions that would otherwise be very difficult to follow.

\smallskip
The article is organized as follows. In the first section, we recall the fundamental principles
of nonstandard analysis, present the basic features of iterated hyper-extensions,
and their connection with tensor products of ultrafilters and partition regularity properties.
In the second section, as a warm-up case,
we use the iterated hyper-extension technique to prove
Van der Waerden's Theorem. In the third section, these ideas
are generalized to prove the strong Hales-Jewett theorem
where one considers substitutions in a variable word of a finite list of words
(not just letters of the alphabet).
We then translate our arguments into
the language of ultrafilters, producing a new ultrafilter proof of that result which,
contrarily to the existing ones, requires neither minimal nor idempotent ultrafilters.

Note that an ultrafilter proof of Central Set
theorem requires minimal idempotent ultrafilters.
Therefore, the existing ultrafilter proof of the strong Hales-Jewett theorem
(see \cite{bbh} and the last paragraph of \S4 in \cite{bl})
requires minimal idempotent ultrafilters.

Finally, the last section contains our main theorem, obtained
applying algebraic properties of spaces of ultrafilters in the nonstandard setting.

\section{Nonstandard analysis and iterated hyper-extensions}

The foundations of the iterated hyper-extensions technique
have been throughly expounded in \cite{dj}, to which we refer the reader for
full details.
Here, we will limit ourselves to stating and proving only those properties
that will be used later in this paper.

We assume the reader is somewhat familiar with nonstandard analysis;
however, to make this article as accessible as possible, we provide
a brief summary of this theory below. The reader
can fill all details by consulting, for example, \cite[Chapter 2]{dgl}, which
provides a (hopefully) friendly introduction.

\subsection{The principles of nonstandard analysis}

The three basic principles of nonstandard analysis are the following:
\begin{itemize}
\item
\textbf{1. Star map}. The \emph{star map} $*:\mathbb{V}\to\mathbb{V}$ is a map defined
from a \emph{universe} $\mathbb{V}$ into itself.
A universe is a sufficiently large collection of sets that contains all the mathematical objects
under study and is closed under all basic set-theoretic constructions, namely
unions, intersections, ordered pairs, Cartesian products, power-sets, sets of functions, \emph{etc.}
It is assumed that ${}^*n=n$ for every $n\in\N$.
The image of an object $X\in\mathbb{V}$ under the star map is denoted by $\hX$ and called
the \emph{hyper-extension} of $X$.
\end{itemize}

Note that since the range is included in the domain, the star map can be iterated,
and so one can consider double hyper-extensions ${}^{**}X={}^*({}^*X)$, triple hyper-extensions
${}^{***}X$, and so forth.

\begin{itemize}
\item
\textbf{2. Transfer principle}.
For every ``elementary" property $P(x_1,\ldots,x_n)$ and for all objects $a_1,\ldots,a_n\in\mathbb{V}$,
one has the equivalence:
$$P(a_1,\ldots,a_n)\ \Longleftrightarrow\ P({}^*a_1,\ldots,{}^*a_n).$$
By ``elementary" we mean any property $P(x_1,\ldots,x_n)$ that refers to the objects $x_1,\ldots,x_n$
and their elements, but where the quantifiers
``for all" and ``there exists" are not allowed
over subsets of the $x_i$ or functions on the $x_i$.
\end{itemize}

The \emph{transfer} principle says that the hyper-extension $\hX$
is ``almost isomorphic" to the original object $X$, in the precise sense
that they satisfy exactly the same elementary properties.
Sometimes, we will refer to the reverse implication
``$P({}^*a_1,\ldots,{}^*a_n)\Longrightarrow P(a_1,\ldots,a_n)$"
as backward \emph{transfer}.

We remark that the iterations of the star map also satisfy the \emph{transfer} principle;
for example, for every elementary property $P(x_1,\ldots,x_n)$ and for all $a_1,\ldots,a_k\in\mathbb{V}$,
one has the equivalences $P(a_1,\ldots,a_n)\Longleftrightarrow P({}^{**}a_1,\ldots,{}^{**}a_k)$,
$P(a_1,\ldots,a_n)\Longleftrightarrow P({}^{***}a_1,\ldots,{}^{***}a_k)$, and so forth.

\begin{itemize}
\item
\textbf{3. $\kappa$-Enlargement property} ($\kappa$ a given cardinal.)
If $\F$ is a nonempty family of sets with the finite intersection property which has cardinality at most $\kappa$,
then $\bigcap_{X\in\F}\hX$ is nonempty.
To our purposes, the cardinality $\kappa=\mathfrak{c}$ of the continuum will be enough.\footnote
{~A nonempty family $\F$ has the finite intersection property if $F_1\cap\ldots\cap F_k\ne\emptyset$
for all choices of finitely many $F_1,\ldots,F_k\in\F$.}
\end{itemize}

We could think of the \emph{enlargement} property as a form of ``saturation" that guarantees the existence of
nonstandard objects, whenever they are ``potentially possible."
For instance, one obtains the existence of positive infinitesimal numbers
by considering the nonempty intersection $\bigcap_{r>0}{}^*(0,r)$ of the hyper-extensions
of intervals of real numbers $(0,r)\subset\R$.

Typically, in nonstandard analysis one works with the hyperreal numbers ${}^*\mathbb{R}$.
Since in this paper we will be dealing with the discrete setting of arithmetic Ramsey theory,
we will focus on the set of natural numbers $\N$ and its iterated hyper-extensions $\hN, {}^{**}\N, {}^{***}\N, \ldots$.
By \emph{transfer}, if $x$ belongs to a set $X$, then its hyper-extension ${}^*x\in\hX$;
since ${}^*n=n$ for all $n\in\N$, we obtain that $\N\subseteq\hN$.
By \emph{enlargement}, it is shown that the inclusion $\N\subsetneq\hN$ is strict,\footnote
{~The family of intervals $\F=\{[n,+\infty)\mid n\in\N\}$ has the finite intersection
property, and so there exist elements $\nu\in\bigcap_{n\in\N}{}^*[n,+\infty)$.
Clearly such $\nu\in\hN\setminus\N$.}
and hence $\hN\subsetneq{}^{**}\N$,
${}^{**}\N\subsetneq{}^{***}\N$, and so forth.

Again by \emph{transfer}, we observe that the \emph{hypernatural numbers} $\hN$ are the positive
part of the discretely ordered ring ${}^*\mathbb{Z}$ of hyperintegers.
Furthermore, $\N$ is an initial segment of $\hN$,
meaning that every element $\nu\in\hN\setminus\N$ is greater than all elements in $\N$.
Consequently, we will refer to the numbers in $\N$ as ``finite" and the numbers
in $\hN\setminus\N$ as ``infinite."

Note that, since $\N$ is an initial segment of $\hN$,
by \emph{transfer} $\hN$ is an initial segment of ${}^{**}\N$;
in particular, ${}^*\mu>\nu$ for all $\mu\in\hN\setminus\N$ and $\nu\in\hN$.
This produces a hierarchy of different levels of infinities
$\nu_1<{}^*\nu_2<{}^{**}\nu_3<{}^{***}\nu_4<\ldots$ for all $\nu_i\in\hN\setminus\N$.

\subsection{The $u$-equivalence}

Let's now see how this is related to ultrafilters, and in particular to tensor products of ultrafilters.\footnote
{~For more details, the reader is refereed to \cite{dn1}.}

To begin with, we observe that every hypernatural number naturally corresponds to a unique ultrafilter.

\begin{definition}
For $\nu\in\hN$, the \emph{generated ultrafilter} $\UU_\nu$ is the
collection of all subsets of $\N$ whose hyper-extension contains $\nu$:
$$\UU_\nu:=\{A\subseteq\N\mid \nu\in\hA\}.$$
\end{definition}

It can be easily verified that $\UU_\nu$ is indeed an ultrafilter on $\N$.
Clearly, for every finite $n\in\N$, the generated ultrafilter is the corresponding principal
ultrafilter, since $\hA\cap\N=A$, and hence $n\in\hA\Leftrightarrow n\in A$ for every $A\subseteq\N$.
If $\nu\in\hN\setminus\N$ is infinite then $\UU_\nu$ is non-principal,
since an infinite point belongs to a hyper-extension $\hA$ only if the set $A$ is infinite.
By $\mathfrak{c}$-\emph{enlargement}, we observe that every non-principal
ultrafilter is generated by an infinite number $\nu$. In fact, \emph{enlargement}
guarantees that there are plenty of generators for any given non-principal ultrafilter
(actually, at least $\mathfrak{c}$); we note that this fact is not a limitation,
but rather a peculiarity that proves useful in studying partition
regularity problems by nonstandard methods.

\begin{definition}
Two numbers $\nu,\mu\in\hN$ are \emph{$u$-equivalent} when the generated
ultrafilters coincide $\UU_\nu=\UU_\mu$, \emph{i.e.},
when $\nu\in\hA\Leftrightarrow\mu\in\hA$ for every $A\subseteq\N$.
In this case we write $\nu\ueq\mu$.
\end{definition}

The notion of generated ultrafilter is extended in the obvious way
to points in hyper-extensions $\hX$ of arbitrary sets $X$.
In particular, if $(\nu_1,\ldots,\nu_k)\in(\hN)^k$,
$$\UU_{(\nu_1,\ldots,\nu_k)}:=\{X\subseteq\N^k\mid (\nu_1,\ldots,\nu_k)\in\hX\}.$$

Generated ultrafilters can also be naturally defined for
elements in iterated hyper-extensions; for example,
if $\alpha\in{}^{***}\N$ then one sets $\UU_\alpha:=\{A\subseteq\N\mid\alpha\in{}^{***}A\}$.
Consequently, the $u$-equivalence relation is extended to elements and tuples
that belongs to iterated hyper-extensions; for example,
$\alpha\in{}^{**}\N$ and $\beta\in{}^{***}\N$ are $u$-equivalent
when $\alpha\in{}^{**}A\Leftrightarrow\beta\in{}^{***}A$ for every $A\subseteq\N$.

The importance of ultrafilters and the $u$-equivalence in Ramsey theory problems
arises from the following characterizations.

\begin{proposition}\label{PRcharacterizations}
Let $\F$ be a nonempty family of nonempty subsets of $X$. The following are equivalent:
\begin{enumerate}
\item
$\F$ is partition regular on the set $X$, \emph{i.e.},
for every finite coloring $X=C_1\cup\ldots\cup C_r$ there exists
$F\in\F$ such that $F\subseteq C_j$ for some $j$.
\item
There exists an ultrafilter $\U$ on $X$ that is a ``witness" of the partition
regularity of $\F$, \emph{i.e.}, for every $B\in\U$ there exists
$F\in\F$ with $F\subseteq B$.
\item
There exists $G\in{}^*\F$ such that all elements in $G$ are $u$-equivalent to each other over $X$,
\emph{i.e.}, for all $\xi,\zeta\in G$ one has $\xi\in\hB\Leftrightarrow\zeta\in\hB$ for every $B\subseteq X$.
\end{enumerate}
\end{proposition}

\begin{proof}
The equivalence $(1)\Leftrightarrow(2)$ is \cite[Theorem 5.7]{hs},
and the equivalence $(1)\Leftrightarrow(3)$ is in the proof of
\cite[Proposition 5.9]{dj}. Although \cite[Proposition 5.9]{dj} is for the
case that $\F$ is a family of finite sets, the proof works for the case
that $\F$ is a family of any nonempty sets.
\end{proof}

\subsection{Tensor products of ultrafilters and iterated hyper-extensions}

Recall that if $\U$ and $\V$ are ultrafilters on $\N$ then
their \emph{tensor product} $\U\otimes\V$ is the ultrafilter on $\N\times\N$
defined by setting for every $X\subseteq\N$:
$$X\in\U\otimes\V\ \Longleftrightarrow\ \{n\in\N\mid\{m\in\N\mid (n,m)\in X\}\in\V\}\in\U.$$

With the usual identification of the Cartesian products $I\times(J\times K)=(I\times J)\times K$,
the operation $\otimes$ is associative, and so we can consider
iterated tensor products $\U_1\otimes\ldots\otimes\U_k$ on $\N^k$.

Here is the fundamental property that connects hyper-extensions to tensor products
of ultrafilters.

\begin{proposition}
Let $\nu,\mu\in\hN$. Then the tensor product $\UU_\nu\otimes\UU_\mu=\UU_{(\nu,{}^*\mu)}$
is generated by the pair $(\nu,{}^*\mu)\in\hN\times{}^{**}\N\subset({}^{**}\N)^2$.
\end{proposition}

\begin{proof}
This is a well-known property, the proof of which can be found in several recent papers.
However, since it is the basis of our applications, we provide a self-contained proof here.

Note that $(\nu,{}^*\mu)\in\hN\times{}^{**}\N\subseteq{}^{**}\N\times{}^{**}\N={}^{**}(\N\times\N)$
belongs to the double hyper-extension of $\N\times\N$.
For every $X\subseteq \N\times\N$,
each of the following properties is equivalent to the next one:
\begin{itemize}
\item
$X\in\UU_{(\nu,{}^*\mu)}$.
\item
$(\nu,{}^*\mu)\in{}^{**}X$.
\item
$\nu\in\{\xi_1\in\hN\mid (\xi_1,{}^*\mu)\in{}^{**}X\}={}^*\{x_1\in\N\mid (x_1,\mu)\in\hX\}$.
\item
$\{x_1\in\N\mid (x_1,\mu)\in\hX\}\in\UU_\nu$.
\item
$\{x_1\in\N\mid \mu\in{}^*\{x_2\in\N\mid (x_1,x_2)\in X\}\}\in\UU_\nu$.
\item
$\{x_1\in\N\mid \{x_2\in\N\mid (x_1,x_2)\in X\}\in\UU_\mu\}\in\UU_\nu$.
\item
$X\in\UU_\nu\otimes\UU_\mu$.
\end{itemize}
\end{proof}

This fact can be used to characterize algebra in the space of ultrafilters $\beta\N$.
In fact, recall that the sum of ultrafilters $\U\oplus\V$ is defined by letting for every $A\subseteq\N$:
$$A\in\U\oplus\V\ \Longleftrightarrow\ \{n\in\N\mid A-n\in\V\}\in\U,$$
where $A-n:=\{m\in\N\mid n+m\in A\}$ is the leftward shift of $A$ by $n$.
If the function $\text{Sum}:\N\times\N\to\N$ is the addition, it is directly verified from the definitions that
$A\in\U\oplus\V\Leftrightarrow \text{Sum}^{-1}(A)\in\U\otimes\V$ for every $A\subseteq\N$.
As a consequence, one obtains:

\begin{proposition}\label{tensorproduct}
Let $\nu,\mu\in\hN$. Then the sum of ultrafilters $\UU_\nu\oplus\UU_\mu=\UU_{\nu+{}^*\mu}$ is
generated by $\nu+{}^*\mu\in{}^{**}\N$.
\end{proposition}

Clearly, the above proposition can be extended to iterated sums.
For example, if $\nu,\mu,\vartheta\in\hN$
then $\UU_\nu\oplus\UU_\mu\oplus\UU_\vartheta=\UU_{\nu+{}^*\mu+{}^{**}\vartheta}$,
and so forth.

It is worth stressing the fact that different ``levels of infinity" are
needed to generate tensor products, and hence sums of ultrafilters.
Indeed, we point out that in general the numbers $\nu+\mu\,\nueq\,\nu+{}^*\mu$ are not $u$-equivalent.

\begin{remark}
The notions and facts seen above about the sum operation $+$ on $\N$ can
equally be considered for arbitrary associative operations; indeed, one can start
from the multiplication $\cdot$ on $\N$ and see that $\UU_\nu\odot\UU_\mu=\UU_{\nu\cdot{}^*\mu}$,
and more generally, one can start from an arbitrary associative operation $\star$ on any set $S$.

Recall that, by a well-known construction, every semigroup $(S,\star)$ is canonically extended to
a compact right-topological semigroup $(\beta S,\ostar)$ on the corresponding space of ultrafilters,
where the associative operation $\ostar$ on the set $\beta S$
is defined by setting for every $A\subseteq S$:\footnote
{~See \cite[Chapter 4]{hs} for all details.}
$$A\in\U\ostar\V\Longleftrightarrow
\{a\in S\mid\{b\in S\mid a\star b\in A\}\in\V\}\in\U\}.$$
Similarly as done with addition, also in this general case one has that:
\begin{itemize}
\item
$A\in\U\ostar\V\Leftrightarrow \star^{-1}(A)=\{(a,b)\mid a\star b\in A\}\in\U\otimes\V$ for every $A\subseteq S$.
\item
$\UU_\alpha\ostar\UU_\beta=\UU_{\alpha\star{}^*\!\beta}$ for all $\alpha,\beta\in{}^*S$.
\end{itemize}
\end{remark}

We denote $\U^{k\ostar}$ for $\underbrace{\U\ostar\U\ostar\cdots\ostar\U}_{k-\mbox{times}}$.

Since we will be dealing with $n$-th iterated hyperextensions for arbitrarily large $n$,
we will sometimes adopt a more convenient notation to avoid excessive ``stars."
To this end, we will use the symbol $\s$ and define inductively:

\smallskip
\begin{itemize}
\item
$\s^0(X)=X$; and $\s^{n+1}(X)={}^*(\s^n(X))$.
\end{itemize}

\smallskip
In particular, for points $\alpha\in\s^m(\N)$, we will consider
the iterated hyper-extensions $\s^n(\alpha)\in\s^n(\s^m(\N))=\s^{n+m}(\N)$.

In the next proposition, we itemize a few fundamental properties connecting iterated
hyper-extensions and $u$-equivalence. We remark that one could easily formulate more general results
(see \cite[\S 5.1]{dj}), but for simplicity we consider here only those facts that will be actually used in the sequel.

\begin{proposition}\label{hyper-extensions}
Let $(S,\star)$ be a semigroup where, similarly to what was done with $\N$, we assume that
${}^*s=s$ for every $s\in S$. Let $\alpha_i\in{}^*S\setminus S$ for $i=1,\ldots,k$. Then:
\begin{enumerate}
\item
$\UU_{(\alpha_1,\s(\alpha_2),\ldots,\s^{k-1}(\alpha_k))}=
\UU_{\alpha_1}\otimes\UU_{\alpha_2}\otimes\cdots\otimes\UU_{\alpha_k}$.
\item
$\UU_{(\alpha_1\star\s(\beta),\s^2(\alpha_2),\ldots,\s^k(\alpha_k))}=
\UU_{\alpha_1\star\s(\beta)}\!\otimes\UU_{\alpha_2}\otimes\cdots\otimes\UU_{\alpha_k}$ for every $\beta\in{}^*S$.
\end{enumerate}
\end{proposition}

\begin{proof}
$(1)$. This is a straight generalization of Proposition \ref{tensorproduct}, which is
the particular case when $k=2$.

$(2)$. This result is implicitly included in \cite{dj}.\footnote
{~See in particular \S 5.1 and Proposition 5.7.}
However, for completeness, we provide a detailed proof here when $k=2$;
the general case can then be proved in the same way, only with a more elaborate notation.

Note that $(\alpha_1\star{}^*\beta,{}^{**}\alpha_2)\in{}^{**}S\times{}^{***}S\subseteq
{}^{***}S\times{}^{***}S$ belongs to the triple hyper-extension of $S\times S$.
For every $X\subseteq S\times S$,
each of the following properties is equivalent to the next one:
\begin{itemize}
\item
$X\in\UU_{(\alpha_1\star{}^*\beta,{}^{**}\alpha_2)}$.
\item
$(\alpha_1\star{}^*\beta,{}^{**}\alpha_2)\in{}^{***}X$.
\item
$\alpha_1\star{}^*\beta\in{}^{**}\{x_1\in S\mid (x_1,\alpha_2)\in\hX\}$.
\item
$\{x_1\in S\mid (x_1,\alpha_2)\in\hX\}\in\UU_{\alpha_1\star{}^*\beta}$.
\item
$\{x_1\in S\mid \alpha_2\in{}^*\{x_2\in S\mid (x_1,x_2)\in X\}\in\UU_{\alpha_1\star{}^*\beta}$.
\item
$\{x_1\in S\mid \{x_2\in S\mid (x_1,x_2)\in X\}\in\UU_{\alpha_2}\}\in\UU_{\alpha_1\star{}^*\beta}$.
\item
$X\in\UU_{\alpha_1\star{}^*\beta}\otimes\UU_{\alpha_2}$.
\end{itemize}
\end{proof}

\begin{corollary}
If $\alpha_i,\alpha'_i\in{}^*S\setminus S$
are such that $\alpha_i\ueq\alpha'_i$ for $i=1,\ldots,k$,
and $\beta,\beta'\in{}^*S$ are such that $\beta\ueq\beta'$, then:
\begin{enumerate}
\item
$(\alpha_1,\s(\alpha_2),\ldots,\s^{k-1}(\alpha_k))\,\ueq\,
(\alpha'_1,\s(\alpha'_2),\ldots,\s^{k-1}(\alpha'_k))$.
\item
$(\alpha_1\star\s(\beta),\s^2(\alpha_2),\ldots,\s^k(\alpha_k))\,\ueq\,
(\alpha'_1\star\s(\beta'),\s^2(\alpha'_2),\ldots,\s^k(\alpha'_k))$.
\item
$\UU_{\alpha_1\star\s(\alpha_2)\star\ldots\star\s^{k-1}(\alpha_k)}=
\UU_{\alpha_1}\ostar\UU_{\alpha_2}\ostar\cdots\ostar\UU_{\alpha_k}$,
\item
$\UU_{\alpha_1\star\s(\beta)\star\s^2(\alpha_2)\star\ldots\star\s^k(\alpha_k)}=
\UU_{\alpha_1}\ostar\UU_\beta\ostar\UU_{\alpha_2}\ostar\cdots\ostar\UU_{\alpha_k}$.
\end{enumerate}
\end{corollary}

\begin{proof}
$(1)$ directly follows from the previous proposition; indeed, if
$\U_i:=\UU_{\alpha_i}=\UU_{\alpha'_i}$
then $(\alpha_1,\s(\alpha_2),\ldots,\s^{k-1}(\alpha_k))$ and $(\alpha'_1,\s(\alpha'_2),\ldots,\s^{k-1}(\alpha'_k))$
both generate the same ultrafilter $\U_1\otimes\cdots\otimes\U_k$.

$(2)$ can be shown by the same way as of $(1)$ that the ultrafilter
generated by the left side is the same as that of the right side.

$(3)$ and $(4)$ follow from the fact that
$\UU_{\alpha_1\ostar\,\s(\beta)}=\UU_{\alpha}\ostar\UU_{\beta}$.
\end{proof}

\section{A proof of Van der Waerden's Theorem}

As a warm-up towards Hales-Jewett Theorem and its multidimensional extension
that will be proved later, we present here a first application of
iterated hyper-extensions and the $u$-equivalence.

\begin{theorem}[Van der Waerden]
Let $\ell\in\N$. For every finite coloring of $\N$
there exists a monochromatic $\ell$-term arithmetic progression
$$a, a+d, a+2d, \ldots, a+(\ell-1)d.$$
\end{theorem}

\begin{proof}
We proceed by induction on $\ell$.
The base cases $\ell=1,2$ are trivial.
To illustrate the inductive step, we see here in detail how to pass from $\ell=3$ to $\ell=4$.
The general inductive step will then be clear.

By the hypothesis, the family $\F=\{\{a, a+d, a+2d\}\mid a,d\in\N\}$ is partition regular on $\N$ and so,
by the nonstandard characterization given by Proposition \ref{PRcharacterizations},
there exist $u$-equivalent hypernatural numbers:
$$a\ \ueq\ a+d\ \ueq\ a+2d.$$
Now let a finite coloring of $\N$ be given; by way of example,
let $\N=C_1\cup C_2\cup C_3$ be a 3-coloring.
Consider the following 4 elements in ${}^{****}\N$:

\medskip
\begin{tabular}{cclclclcl}
$\xi_1$ & $=$ & $a+3d$ & $+$ & ${}^*a$  & $+$ &  ${}^{**}a$ & $+$ & ${}^{***}a$
\\
$\xi_2$ & $=$  & $a+3d$ & $+$ & ${}^*a+3\,{}^*\!d$ & $+$ & ${}^{**}a$ & $+$ & ${}^{***}a$
\\
$\xi_3$ & $=$ & $a+3d$ & $+$ & ${}^*a+3\,{}^*\!d$ & $+$ & ${}^{**}a+3\,{}^{**}\!d$ & $+$ & ${}^{***}a$
\\
$\xi_4$ & $=$ & $a+3d$ & $+$ & ${}^*a+3\,{}^*\!d$ & $+$ & ${}^{**}a+3\,{}^{**}\!d$ & $+$ & ${}^{***}a+3\,{}^{***}\!d$
\end{tabular}

\medskip
Since ${}^{****}\N={}^{****}C_1\cup{}^{****}C_2\cup{}^{****}C_3$,
by the pigeonhole principle there must be two elements of the same
color, say $\xi_2, \xi_4\in{}^{****}C_j$.
Because of the different ``levels", by
Proposition \ref{hyper-extensions} we have that the elements
$\xi_2, \xi_2', \xi_2''$ below are $u$-equivalent to each other,
and hence the following four elements of ${}^{****}\N$ belong to the same color ${}^{****}C_j$:

\medskip
\begin{tabular}{cclclclcl}
$\xi_2$ & $=$  & $a+3d$ & $+$ & ${}^*a+3\,{}^*\!d$ & $+$ & ${}^{**}a$ & $+$ & ${}^{***}a$
\\
$\xi'_2$ & $=$  & $a+3d$ & $+$ & ${}^*a+3\,{}^*\!d$ & $+$ & ${}^{**}a+{}^{**}\!d$ & $+$ &
${}^{***}a+{}^{***}\!d$
\\
$\xi''_2$ & $=$  & $a+3d$ & $+$ & ${}^*a+3\,{}^*\!d$ & $+$ & ${}^{**}a+2\,{}^{**}\!d$ & $+$ &
${}^{***}a+2\,{}^{***}\!d$
\\
$\xi_4$ & $=$ & $a+3d$ & $+$ & ${}^*a+3\,{}^*\!d$ & $+$ & ${}^{**}a+3\,{}^{**}\!d$ & $+$ &
${}^{***}a+3\,{}^{***}\!d$
\end{tabular}

\medskip
We observe that $\xi_2, \xi'_2, \xi''_2, \xi_4\in{}^{****}C_j$ form an arithmetic progression
of common difference ${}^{**}\!d+{}^{***}\!d$.
By backward \emph{transfer}, we obtain the existence of a monochromatic $4$-term
arithmetic progression in $C_j$.
\end{proof}

\begin{remark}
A translation of the above nonstandard proof
in terms of iterated sums $\U\oplus\ldots\oplus\U$ of non-principal ultrafilters $\U$ on $\N$
was presented in \cite{dn3}. The reader may find
it interesting to compare the two proofs, which
formalize the same ideas into two different languages,
namely the language of nonstandard analysis and the language of ultrafilters,
respectively.
\end{remark}

\section{A new proof of Hales-Jewett Theorem}\label{sec-HJ}

Hales-Jewett Theorem is a classic result in Ramsey theory that
can be seen as a generalization of Van der Waerden's theorem
in the framework of words.

Recall that a \emph{word} is a finite string $w=a_1\cdots a_n$ of \emph{letters} $a_i\in A$
taken from a nonempty set $A$, called the \emph{alphabet}.
The number of letters that are used to form the word $w$
is called the \emph{length} of $w$, and denoted $\ell(w)$.
The set of all words over the alphabet $A$ is denoted $W_A$.
We often write $W$ for $W_A$ when $A$ is clear from the context.

The \emph{concatenation} of two words $w=a_1\cdots a_n$ and
$w'=a'_1\cdots a'_m$ is the word $w\star w'=a_1\cdots a_n a'_1\cdots a'_m$
obtained by joining the two strings end-to-end to form a single string.
Clearly, concatenation $\star$ is an associative operation, but it is not commutative,
except when $A=\{a\}$ contains a single letter.

A \emph{variable word} over $A$ is a word $w(v)\in W_{A\cup\{v\}}\setminus W_A$ over
the alphabet $A\cup\{v\}$ where the \emph{variable} $v\notin A$ actually occurs.
If $w(v)$ is a variable word over $A$ and $a\in A$ is a letter, then one can consider
the word $w(a)\in W_A$ obtained by replacing each
occurrence of the variable $v$ in $w(v)$ by $a$.
Such a word $w(a)$ is called an \emph{instance} of the variable word $w(v)$.

\begin{remark}\label{singleletter}
When the alphabet $A=\{a\}$ consists of a single letter,
every word $w\in W_A$ has the form $w=a^n:=a\star\ldots\star a$ ($n$-times) and so
it is uniquely determined by its length.
In this case $W_A$ can be identified with the set of natural numbers.
\end{remark}

\begin{theorem}[Hales-Jewett]
Let $W=W_A$ be the set of words over a finite alphabet $A$.
For every finite coloring of $W$ there exists a variable word $w(v)$ such that
the set of instances $w[A]:=\{w(a)\mid a\in A\}$ is monochromatic.
\end{theorem}

More generally, given a variable word $w(v)$ and a word $b\in W$
(not necessarily a letter of the alphabet),
one can consider the instance $w(b)\in W$ obtained by replacing each
occurrence of the variable $v$ in $w(v)$ by $b$. Note that
the length of $w(v)$ and the length of $w(b)$ are different
unless $b$ has length $1$. The previous theorem also holds for this extended notion of instance.

\begin{theorem}[Hales-Jewett -- strong version]\label{HJstrong}
Let $W=W_A$ be the set of words over a finite alphabet $A$,
and let $B\subset W$ be a finite set. For every finite coloring of $W$
there exists a variable word $w(v)$ such that
the set of instances $w[B]:=\{w(b)\mid b\in B\}$ is monochromatic.
\end{theorem}

An ultrafilter proof of this last result was obtained by
V.~Bergelson, A.~Blass, and N.~Hindman \cite{bbh} as a particular case
of stronger variable-word partition theorems (see also \cite{bl}).
Their proof makes an essential use of minimal idempotent ultrafilters.

\smallskip
We observe that the latter Hales-Jewett Theorem formulation is actually stronger than then former,
already when the alphabet $A=\{a\}$ consists of a single letter;
for instance, in this case the former formulation gives a trivial property,
while the latter amounts to Van der Waerden's Theorem.
Let us see this in detail.

Recall that, as noticed above, one can identify each word $a^n:=a\star\cdots\star a$ ($n$-times) in $W_{\{a\}}$
with its length $\ell(a^n)=n\in\N$. Given a finite coloring $\N=C_1\cup\ldots\cup C_r$ and $k\in\N$,
consider the corresponding finite coloring $W_{\{a\}}=C'_1\cup\ldots C'_r$
where $C'_s:=\{w\in W_A\mid\ell(w)\in C_i\}$, and take
the finite set of words $\{a,a^2,\ldots,a^k\}$. By the strong version
of Hales-Jewett Theorem, there exists
a variable word $w(v)$ over $\{a\}$ such that the set of instances
$\{w(a),w(a^2),\ldots,w(a^k)\}\subseteq C'_j$ is monochromatic.
Finally, observe that $\ell(w(a)),\ldots,\ell(w(a^k))\in C_j$ is
a monochromatic arithmetic progression of length $k$;
precisely, if $d$ is the number of times that the variable $v$ appears in $w(v)$,
and $c:=\ell(w(v))-d$, then $\ell(w(a^i))=c+id$ for every $i=1,\ldots,k$.

Similarly, it is also easily seen that the converse property holds,
\emph{i.e.}, Van der Waerden's Theorem directly implies
the strong version of Hales-Jewett Theorem where the alphabet
$A=\{a\}$ is a singleton.

\begin{proof}[Proof of Theorem \ref{HJstrong}]
Without loss of generality, we can assume that ${}^*a=a$ for every letter $a\in A$,
and hence $\hA=A$, since $A$ is finite.
Note also that ${}^*b=b$ for every word $b\in W$.
The hyper-extension ${}^*W$
is the set of \emph{internal words} on $A$ consisting of finite words and
\emph{hyperfinite words}, that is, the set of internal strings
$\w=a_1\cdots a_\nu$ of letters $a_i\in\hA=A$
of finite or hyperfinite length $\ell(\w)=\nu\in\hN$.\footnote
{~In other words, an element $\w\in{}^*W$ is an internal function $\w:\{1,\ldots,\nu\}\to A$
where $\nu\in\hN$, and the length $\ell(\w)=\nu$.}
A hyperfinite word is also called a \emph{hyperword}.
Clearly, $W\subsetneq {}^*W$; in fact, an internal word $\w$ belongs to
$W$ if and only if its length $\ell(\w)\in\N$ is finite.

We will assume that also the variable $v$ is such that ${}^*v=v$, so that the hyper-extension
${}^*(W_{\{v\}}\setminus W)={}^*W_{\{v\}}\setminus{}^*W$ is
the set of \emph{variable hyperwords} on $A$, where
$W_{\{v\}}$ means $W_{A\cup\{v\}}$.
Similarly to the finite case, if $\w(v)$ is a variable hyperword, for every $b\in{}^*W$
(and hence for every $b\in W$) one can consider the hyperword
$\w(b)\in {}^*W$ obtained by replacing each occurrence of $v$ in $\w(v)$ by $b$.\footnote
{~To be precise, this is obtained by considering the hyper-extension of
the ``substitution function" $S:W(v)\times W\to W$ where $(w(v),b)\mapsto w(b)$.}

We will prove the theorem proceeding by induction
on the cardinality $n$ of the set of words $B=\{b_1,\ldots,b_n\}\subset W$.
The base case $n=1$ is trivial.

Let us now consider the inductive step.
For simplicity, and to better illustrate the idea, as an example we first present in detail
the step from $n=3$ to $n=4$, in the case of $5$-colorings of $W$.

Let $B=\{b_1, b_2, b_3, b_4\}\subset W$. By the inductive hypothesis for $n=3$,
the nonstandard characterization given by Proposition \ref{PRcharacterizations}
yields the existence of a variable hyperword $\w(v)\in{}^*W_{\{v\}}$
such that
$\w(b_1)\ueq\w(b_2)\ueq\w(b_3)$.

Given a $5$-coloring of $W=C_1\cup\ldots\cup C_5$,
consider the following 6 elements in the 6-th hyper-extension
${}^{******}W$,
where $\star$ denotes the concatenation of words:

\medskip
\begin{tabular}{lclclclclclcl}
$\xi_0\!\!$ & $\!\!:=\!\!$ & $\!\!\w(b_1)\!\!$ & $\star$ & $\!\!{}^*\w(b_1)\!\!$
& $\star$ &  $\!\!{}^{**}\w(b_1)\!\!$ & $\star$ & $\!\!{}^{***}\w(b_1)\!\!$
& $\star$ &  $\!\!{}^{****}\w(b_1)\!\!$ & $\star$
& $\!\!{}^{*****}\w(b_4)$
\\
$\xi_1\!\!$ & $\!\!:=\!\!$  & $\!\!\w(b_1)\!\!$ & $\star$ & $\!\!{}^*\w(b_1)\!\!$
& $\star$ &  $\!\!{}^{**}\w(b_1)\!\!$ & $\star$ & $\!\!{}^{***}\w(b_1)\!\!$
& $\star$ &  $\!\!{}^{****}\w(b_4)\!\!$ & $\star$
& $\!\!{}^{*****}\w(b_4)$
\\
$\xi_2\!\!$ & $\!\!:=\!\!$ & $\!\!\w(b_1)\!\!$ & $\star$ & $\!\!{}^*\w(b_1)\!\!$
& $\star$ &  $\!\!{}^{**}\w(b_1)\!\!$ & $\star$
& $\!\!{}^{***}\w(b_4)\!\!$
& $\star$ &  $\!\!{}^{****}\w(b_4)\!\!$ & $\star$
& $\!\!{}^{*****}\w(b_4)$
\\
$\xi_3\!\!$ & $\!\!:=\!\!$ & $\!\!\w(b_1)\!\!$ & $\star$ & $\!\!{}^*\w(b_1)\!\!$
& $\star$ &  $\!\!{}^{**}\w(b_4)\!\!$ & $\star$
& $\!\!{}^{***}\w(b_4)\!\!$
& $\star$ &  $\!\!{}^{****}\w(b_4)\!\!$ & $\star$
& $\!\!{}^{*****}\w(b_4)$
\\
$\xi_4\!\!$ & $\!\!:=\!\!$ & $\!\!\w(b_1)\!\!$ & $\star$ & $\!\!{}^*\w(b_4)\!\!$
& $\star$ &  $\!\!{}^{**}\w(b_4)\!\!$ & $\star$ & $\!\!{}^{***}\w(b_4)\!\!$
& $\star$ &  $\!\!{}^{****}\w(b_4)\!\!$ & $\star$ & $\!\!{}^{*****}\w(b_1)$
\\
$\xi_5\!\!$ & $\!\!:=\!\!$ & $\!\!\w(b_4)\!\!$ & $\star$ & $\!\!{}^*\w(b_4)\!\!$
& $\star$ &  $\!\!{}^{**}\w(b_4)\!\!$ & $\star$ & $\!\!{}^{***}\w(b_4)\!\!$
& $\star$ &  $\!\!{}^{****}\w(b_4)\!\!$ & $\star$
& $\!\!{}^{*****}\w(b_4)$
\end{tabular}

\medskip
By the pigeonhole principle, in the coloring
${}^{******}W={}^{******}C_1\cup\ldots\cup{}^{******} C_r$ there must be
two elements of the same color;
by way of example, say $\xi_1,\xi_3\in{}^{******}C_j$.
Since $\w(b_1)\ueq\w(b_2)\ueq\w(b_3)$,
the four elements $\xi_1=\zeta_1\ueq\ldots\ueq\zeta_4=\xi_3$
below are $u$-equivalent to each other,
and hence they all belong to ${}^{******}C_j$:

\medskip
\begin{tabular}{lclclclclclcl}
$\zeta_1\!\!$ & $\!\!:=\!\!$  & $\!\!\w(b_1)\!\!$ & $\star$ & $\!\!{}^*\w(b_1)\!\!$
& $\star$ &  $\!\!{}^{**}\w(b_1)\!\!$ & $\star$ & $\!\!{}^{***}\w(b_1)\!\!$
& $\star$ &  $\!\!{}^{****}\w(b_4)\!\!$ & $\star$
& $\!\!{}^{*****}\w(b_4)$
\\
$\zeta_2\!\!$ & $\!\!:=\!\!$ & $\!\!\w(b_1)\!\!$ & $\star$ & $\!\!{}^*\w(b_1)\!\!$
& $\star$ &  $\!\!{}^{**}\w(b_2)\!\!$ & $\star$
& $\!\!{}^{***}\w(b_2)\!\!$
& $\star$ &  $\!\!{}^{****}\w(b_4)\!\!$ & $\star$
& $\!\!{}^{*****}\w(b_4)$
\\
$\zeta_3\!\!$ & $\!\!:=\!\!$ & $\!\!\w(b_1)\!\!$ & $\star$
& $\!\!{}^*\w(b_1)\!\!$  & $\star$ &  $\!\!{}^{**}\w(b_3)\!\!$
& $\star$ & $\!\!{}^{***}\w(b_3)\!\!$
& $\star$ &  $\!\!{}^{****}\w(b_4)\!\!$ & $\star$
& $\!\!{}^{*****}\w(b_4)$
\\
$\zeta_4\!\!$ & $\!\!:=\!\!$ & $\!\!\w(b_1)\!\!$ & $\star$ & $\!\!{}^*\w(b_1)\!\!$
& $\star$ &  $\!\!{}^{**}\w(b_4)\!\!$ & $\star$
& $\!\!{}^{***}\w(b_4)\!\!$
& $\star$ &  $\!\!{}^{****}\w(b_4)\!\!$ & $\star$
& $\!\!{}^{*****}\w(b_4)$
\end{tabular}

\medskip
Now consider the following variable hyperword
$\widetilde{\w}(v)\in{}^{******}(W_{A\cup\{v\}})$:
$$\widetilde{\w}(v)\ :=\ \w(b_1)\star{}^*\w(b_1)\star
{}^{**}\w(v)\star{}^{***}\w(v)\star{}^{****}\w(b_4)\star{}^{*****}\w(b_4).$$

Then $\widetilde{\w}(b_1)=\zeta_1$, $\widetilde{\w}(b_2)=\zeta_2$,
$\widetilde{\w}(b_3)=\zeta_3$, and $\widetilde{\w}(b_4)=\zeta_4$, and so the set
$\{\widetilde{\w}(b_1), \widetilde{\w}(b_2), \widetilde{\w}(b_3), \widetilde{\w}(b_4)\}\subseteq{}^{******}C_j$
is monochromatic. Finally, by backward \emph{transfer}, we obtain the existence of
a variable word $w(v)\in W_{\{v\}}$ such that set of instances
$\{w(b_1), w(b_2), w(b_3), w(b_4)\}\subseteq C_j$
is monochromatic with respect to the initial $5$-coloring $W=C_1\cup\ldots\cup C_5$.

While the idea should be clear by now, for completeness let us
write down the inductive step from $n$ to $n+1$ in its general form.

Let $B=\{b_1,\ldots,b_n,b_{n+1}\}$.
By the nonstandard characterization of the inductive hypothesis,
there exists a variable hyperword $\w(v)\in{}^*W_{\{v\}}\setminus {}^*W$  such that
$$\w(b_1)\ueq\ldots\ueq\w(b_n).$$
Given a finite coloring of $W=C_1\cup\ldots\cup C_r$,
consider the following elements for $s=0,1,\ldots,r$:\footnote
{~We agree that $\xi_r=\prod_{i=0}^r\s^i(\w(b_{n+1}))$.}
\begin{itemize}
\item
$\xi_s:=\prod_{i=0}^{r-s-1}\s^i(\w(b_1))\star\prod_{i=r-s}^r\s^i(\w(b_{n+1}))\in\s^{r+1}(W)$.
\end{itemize}

\medskip
By the pigeonhole principle, in the coloring
$\s^{r+1}(W)=\s^{r+1}(C_1)\cup\ldots\cup\s^{r+1}(C_r)$ there must be
$s<t$ such that the elements $\xi_s, \xi_t\in\s^{r+1}(C_j)$ have the same color.
Since the elements $\w(b_i)$ are $u$-equivalent to each other for $i=1,\ldots,n$,
also the elements
$\xi_s=\zeta_1\,\ueq\,\ldots\,\ueq\zeta_n\,\ueq\,\zeta_{n+1}=\xi_t$ below are
$u$-equivalent to each other, and hence they all belong to the same color $\s^{r+1}(C_j)$:
\begin{itemize}
\item
$\zeta_k:=\prod_{i=0}^{r-t-1}\s^i(\w(b_1))\star\prod_{i=r-t}^{r-s-1}\s^i(\w(b_k))\star
\prod_{i=r-s}^r\s^i(\w(b_{n+1}))\in\s^{r+1}(W)$
\end{itemize}

Now consider the following variable hyperword
 $\widetilde{\w}(v)\in\s^{r+1}(W_{A\cup\{v\}})$:
\begin{itemize}
\item
$\widetilde{\w}(v):=\prod_{i=0}^{r-t-1}\s^i(\w(b_1))\star\prod_{i=r-t}^{r-s-1}\s^i(\w(v))\star
\prod_{i=r-s}^r\s^i(\w(b_{n+1}))$

$\in\s^{r+1}(W_{\{v\}})$.
\end{itemize}

Then $\widetilde{\w}(b_k)=\zeta_k$ for $k=1,\ldots,n,n+1$ and so the set
$$\{\widetilde{\w}(b_1),  \ldots, \widetilde{\w}(b_n), \widetilde{\w}(b_{n+1})\}\subseteq\s^{r+1}(C_j)$$
is monochromatic. Finally, by backward \emph{transfer}, we obtain the existence of
a variable word $w(v)\in W_{\{v\}}$ such that set of instances
$\{w(b_i)\mid i=1,\ldots,n,n+1\}\subseteq C_j$
is monochromatic with respect to the initial coloring $W=C_1\cup\ldots\cup C_r$.
\end{proof}

Below we present a translation of the above nonstandard arguments into the
language of ultrafilters. It is worth noting that the resulting ultrafilter proof,
unlike existing ones, relies exclusively on non-principal ultrafilters
and uses neither minimal nor idempotent ultrafilters.

\begin{proof}[Ultrafilter proof of Theorem \ref{HJstrong}]
We proceed by induction on the cardinality $n$ of the set of words $B=\{b_1,\ldots,b_n\}\subset W$.
The base case $n=1$ is trivial.
Now assume that Theorem \ref{HJstrong} is true for $B=\{b_1,\ldots,b_n\}\subseteq W$;
we want to show that the theorem is also true for $B'=\{b_1,\ldots,b_n,b_{n+1}\}\subseteq W$.

For every $A\subseteq W$ let
$\Lambda(A):=\{w(x)\in W(x)\mid w[B]\subseteq A\ \text{or}\ w[B]\cap A=\emptyset\}$.
The family $\F:=\{\Lambda(A)\mid A\subseteq W\}$ has the finite intersection property.
Indeed, given $A_1,\ldots,A_k\subseteq W$, pick a finite partition
$W=D_1\cup\ldots\cup D_s$ such that $A_i\subseteq D_j$ or $A_i\cap D_j=\emptyset$ for
every $i=1,\ldots,k$ and for every $j=1,\ldots,s$.
By the hypothesis there exists a piece $D_j$ and a variable word $w(v)$ such that
$w[B]\subseteq D_j$; then clearly $w(v)\in\bigcap_{i=1}^k\Lambda(A_k)\ne\emptyset$.

Now pick an ultrafilter $\W$ on $W(v)$ that extends the family $\F$.
For $b\in W$ let $S_b:W(v)\to W$ be the ``substitution" function $S_b:w(x)\mapsto w(b)$.
We observe that all image ultrafilters $S_{b_i}(\W)=S_{b_j}(\W)$
are equal for $1\le i,j\le n$. Indeed, if $A\in S_{b_i}(\W)$, that is, if $S_{b_i}^{-1}(A)\in\W$,
also then $A\in S_{b_j}(\W)$, since
$S_{b_j}^{-1}(A)\supseteq\{w(x)\in W(x)\mid w[B]\subseteq A\}=
S_{b_i}^{-1}(A)\cap\Lambda(A)\in\W$.

Let $\U:=S_{b_1}(\W)=\ldots=S_{b_n}(\W)$ and let $\V:=S_{b_{n+1}}(\W)$.
Given a finite coloring $W=C_1\cup\ldots\cup C_r$, consider the following $r+1$ ultrafilters:
\begin{itemize}
\item
$\X_s:=\U^{s\ostar}\ostar\V^{(r+1-s)\ostar}$ for $s=0,1,\ldots,r$.\footnote
{~We agree that $\X_0=\V^{(r+1)\ostar}$.}
\end{itemize}
For every $s$ there exists $i_s$ such that $C_{i_s}\in\X_s$ and so, by the pigeonhole principle,
there exist $0\le s<t\le r$ and a color $C=C_i\in\X_s\cap\X_t$.
Note that $\X_s=\U'\ostar\V^{k\ostar}\ostar\V'$ and $\X_t=\U'\ostar\U^{k\ostar}\ostar\V'$
where $\U':=\U^{s\ostar}$; $k=t-s$, and $\V'=\V^{(r+1-t)\ostar}$.
Now let $D:=\{w\in W\mid\{w'\in W\mid w\star w'\in C\}\in\V'\}$.
Since $C\in\X_s$ we have that $D\in\U'\ostar\V^{k\ostar}$; and since
$C\in\X_t$ we have that $D\in\U'\ostar\U^{k\ostar}$.
Then $D_1:=\{w\in W\mid\{w'\in W\mid w\star w'\in D\}\in\V^{k\ostar}\}\in\U'$
and $D_2:=\{w\in W\mid\{w'\in W\mid w\star w'\in D\}\in\U^{k\ostar}\}\in\U'$.
Now pick any word $\overline{w}\in D_1\cap D_2\in\U'$.
Then $E:=\{w'\in W\mid \overline{w}\star w'\in D\}\in\U^{k\ostar}\cap\V^{k\ostar}$.

For $b\in W$ let $\sigma_b:[W(v)]^k\to\W$ be the function where
$\sigma_b:(w_1(v),\ldots,w_k(v))\mapsto w_1(b)\star\ldots\star w_k(b)$.
For $i=1,\ldots,n$, we observe that $E\in\U^{k\ostar}=S_{b_i}(\W)^{k\ostar}=\sigma_{b_i}(\W^{k\otimes})$
where $\W^{k\otimes}=\W\otimes\ldots\otimes\W$ is the $k$-th iterated tensor product of $\W$ with itself;
and similarly, $E\in\V^{k\ostar}=S_{b_{n+1}}(\W)^{k\ostar}=\sigma_{b_{n+1}}(\W^{k\otimes})$.
Then we can pick a $k$-tuple of variable words
$(u_1(x),\ldots,u_k(x))\in\bigcap_{i=1}^{n+1}\sigma_{b_i}^{-1}(E)\in\W^{k\otimes}$.
This means that for every $i=1,\ldots,n+1$ one has
$u_1(b_i)\star\ldots\star u_k(b_i)\in E$, and hence
$w_i:=\overline{w}\star u_1(b_i)\star\ldots\star u_k(b_i)\in D$.
Then the sets $G_i:=\{w'\in W\mid w_i\star w'\in C\}\in\V'$ for all $i=1,\ldots,n+1$,
and so we can pick a word $w'\in\bigcap_{i=1}^{n+1}G_i\in\V'$.
Finally, if we let $\widetilde{w}(x):=\overline{w}\star u_1(x)\star\ldots\star u_k(x)\star w'$
then $\widetilde{w}(b_i)\in C$ for every $i=1,\ldots,n+1$.

\end{proof}

\section{The space of Hales-Jewett witnesses}

\begin{definition}
An ultrafilter $\U$ on $W$ is a \emph{Hales-Jewett witness}, or simply a \emph{HJ-witness},
if for every $X\in\U$ and for all $b_1,\ldots,b_n\in W$ there exists
a variable word $w(v)$ such that all instances $w(b_1),\ldots,w(b_n)\in X$.
\end{definition}

To prove our main result
we will need a ``special" ultrafilter, namely a HJ-witness that is idempotent
with respect to an appropriate operation.

Let us briefly recall a few algebraic notions and facts.
All details and proofs can be found in \cite[Chapter 4]{hs}.

Let $(S,\star)$ be a semigroup. A nonempty set $L\subseteq S$ is a \emph{left ideal}
if for every $a\in S$ and for every $b\in L$ one has $a\star b\in L$.
The notion of \emph{right ideal} is defined similarly.
A \emph{bilateral} ideal is a left ideal that is also a right ideal.
Note that if $I$ is a left or a right ideal then $(I,\star)$ is a sub-semigroup of $(S,\star)$.

A semigroup $(S,\star)$ where $S$ is a topological space is named \emph{right-topological}
if for every $b\in S$, the function $\psi_b:a\mapsto a\star b$ is continuous.
A fundamental result in this area is the following.
\begin{itemize}
\item
\emph{\textbf{Ellis' Lemma}: In every compact right-topological semigroup $(S,\star)$ there exist
idempotent elements $a\star a=a$.}
\end{itemize}

Every semigroup $(S,\star)$ is canonically extended to a compact right-topological
semigroup $(\beta S,\ostar)$ on the corresponding space of ultrafilters,
where the operation $\ostar$ is defined by setting
for all $\U,\V\in\beta S$ and for every $X\subseteq S$:
$$X\in\U\ostar\V\Longleftrightarrow
\{s\in S\mid\{t\in S\mid s\star t\in X\}\in\V\}\in\U\}.$$

The following properties hold:
\begin{itemize}
\item
\emph{Every left ideal of $(\beta S,\ostar)$ includes a minimal left ideal with respect to inclusion.}
\item
\emph{Every minimal left ideal $L$ of $(\beta S,\ostar)$ is closed,
and hence $(L,\ostar)$ is a compact right-topological (sub)semigroup.}
\item
\emph{The union of all minimal left ideals of $(\beta S,\ostar)$
forms a bilateral ideal denoted $K(\beta S,\ostar)$, and
named the ``smallest ideal" because it is included in every other bilateral ideal.}
\end{itemize}

Ultrafilters in $K(\beta S,\ostar)$ are called \emph{minimal ultrafilters},
and satisfy relevant combinatorial properties; for example,
if $\U$ is a minimal ultrafilter of $(\beta\N,\oplus)$,
then every set $A\in\U$ contains arbitrarily
large arithmetic progressions. Combining the above results, one obtains
the following property:
\begin{itemize}
\item[$(\dagger)$]
\emph{For every left ideal $L$ of $(\beta S,\ostar)$ there exists
an ultrafilter $\U\in L$ that is minimal and idempotent: $\U\ostar\U=\U$.}
\end{itemize}

Let $W=W_A$ be the set of words over an alphabet $A$.
We already observed that the concatenation of words
$w\star w'$ is an associative operation, and
hence $(W,\star)$ is a (non-abelian) semigroup.
Our desired HJ-witness will be guaranteed by
the algebraic properties of the right-topological semigroup $(\beta W,\ostar)$.

\begin{proposition}\label{bilateral}
The space $\HJ=\{\U\in\beta W\mid \U\ \text{is a HJ-witness}\}$ is a
nonempty closed bilateral ideal of $(\beta W,\ostar)$, and so
the smallest ideal $K(\beta W,\ostar)\subseteq\HJ$.
In particular, there exist idempotent ultrafilters $\U=\U\ostar\U\in\HJ$
that are minimal in $(\beta W,\ostar)$.
\end{proposition}

\begin{proof}
By the strong version of Hales-Jewett Theorem we know that
for every finite set $B\subset W$, the family
$\F_B:=\left\{\{w(b)\mid b\in B\}\mid w(v)\in W_{\{v\}}\right\}$
is partition regular on $W$. So, by the characterization of
\ref{PRcharacterizations}, the following space of ultrafilters is nonempty:
$$\HJ_B:=\{\U\in\beta W\mid \forall X\in\U\ \exists w(v)\in W_{\{v\}}\ \text{such that}\
\{w(b)\mid b\in B\}\subseteq X\}.$$

It is easily verified that the subspaces $\HJ_B\subseteq\beta W$ are closed.
Now observe that for all $B_1,\ldots,B_k\subset W$, the intersection
$\HJ_{B_1}\cap\ldots\cap\HJ_{B_k}\supseteq\HJ_B\ne\emptyset$ where
$B:=B_1\cup\ldots\cup B_k$.
So, by compactness of $\beta W$, it follows that the whole intersection
$\HJ=\bigcap\{\HJ_B\mid B\subset W\ \text{finite}\}\ne\emptyset$ is closed and nonempty.

We have to verify that $\HJ$ is a bilateral ideal of $(\beta W,\ostar)$.
Fix $b_1,\ldots,b_n\in W$. Let $X\in\U\ostar\V$, that is,
$$\Lambda:=\{w\in W\mid \{w'\in W\mid w\star w'\in X\}\in\V\}\in\U.$$

Assume first that $\U\in\HJ$.
Then there exists a variable word $w(x)$
such that for every $i=1,\ldots,n$, one has $w(b_i)\in\Lambda$, and hence
$\Gamma(b_i):=\{w'\in W\mid w(b_i)\star w'\}\in\V$.
Pick any word $w'\in\bigcap_{i=1}^n\Gamma(b_i)\in\V$, and let
$\widetilde{w}(v):=w(x)\star w'\in W(w)$. Then all instances
$\widetilde{w}(b_i)=w(b_i)\star w'\in X$. This shows that $\U\ostar\V\in\HJ$.

Now assume that $\V\in\HJ$. Pick any word $b\in\Lambda$.
Then $\Gamma(b):=\{w'\in W\mid b\star w'\in X\}\in\V$
and so there exists a variable word $w(x)$ such that $w(b_i)\in \Gamma(b)$
for all $i=1,\ldots,n$. If $\widetilde{w}(x):=b\star w(v)\in W_{\{v\}}$,
then all instances $\widetilde{w}(b_i)=b\star w(b_i)\in X$, and also in
this case we can conclude that $\U\ostar\V\in\HJ$.
\end{proof}

\begin{remark}
Note that a minimal ultrafilter of the compact right-topological subsemigroup $(\HJ,\ostar)$
may not be minimal in $(\beta W,\ostar)$, and so the last part of the previous theorem
requires the fact that $\HJ$ is a closed left ideal (actually bilateral),
not just a closed subsemigroup.
\end{remark}

We will use the nonstandard counterpart of the existence of HJ-witnesses that
are idempotent. (We do not need minimality.)

\begin{lemma}\label{lemma}
There exists a variable hyperword $\w(x)\in{}^*W_{\{v\}}\setminus W_{\{v\}}$ such that:
\begin{enumerate}
\item
The instances $\{\w(b)\mid b\in W\}$ are $u$-equivalent to each other.
\item
$\w(b)\ueq\w(b')\star{}^*\w(b'')$ for all $b,b',b''\in W$.
\end{enumerate}
\end{lemma}

\begin{proof}
By the previous Proposition \ref{bilateral}
we can pick an idempotent ultrafilter $\U=\U\ostar\U\in\HJ$.
For every $b\in W$ and for every $X\in\U$, let
$$\Theta(b,X):=\{w(v)\in W_{\{v\}}\mid w(b)\in X\}.$$

We observe that the family $\{\Theta(b,X)\mid b\in W,\, X\in\U\}$ has the finite intersection property.
Indeed, let $b_1,\ldots,b_n\in W$ and $X_1,\ldots,X_n\in\U$ be given.
Since $\U\in\HJ$ and $Y:=X_1\cap\ldots\cap X_n\in\U$, there exists a variable word $w(v)$
such that $w(b_i)\in Y$ for every $i=1,\ldots,n$.
Clearly, such a variable word $w(v)\in\bigcap_{i=1}^n\Theta(b_i,X_i)$.

By the \emph{enlargement property}, the intersection of the hyper-extensions
$$\bigcap_{b\in W,\, X\in\U}{}^*\Theta(b,X)\ne\emptyset.$$

Pick a variable hyperword $\w(v)\in{}^*W_{\{v\}}$ in that intersection;
this means that for every $b\in W$ and for every $X\in\U$ one has $\w(b)\in{}^*X$.
Then the elements in $\{\w(b)\mid b\in W\}$ are $u$-equivalent to each other, since they all
generate $\UU_{\w(b)}=\U$ the same ultrafilter.
Note that $\w(v)\notin W_{\{v\}}$, otherwise every $\w(b)\in W$;
this is not possible because if $b\ne b'$ then $\w(b)\ne\w(b')$,
and two different words of $W$ cannot be $u$-equivalent.\footnote
{~In fact, it is readily verified that no idempotent ultrafilter $\U=\U\ostar\U$ on $W$ is principal.}
Finally, since $\U$ is idempotent, for all $b,b',b''\in W$ we have
$\UU_{\w(b)}=\U=\U\ostar\U=\U_{\w(b')}\ostar\U_{\w(b'')}=\UU_{\w(b')\star{}^*\w(b'')}$,
and hence $\w(b)\ueq\w(b')\star{}^*\w(b'')$.
\end{proof}

\section{The main result}

We are finally ready to prove our main result, a multidimensional
generalization of Hales-Jewett Theorem that simultaneously extends
in a ``direct" way also Ramsey and Hindman's Theorems. Let us first fix notation.

If $w_1,\ldots,w_n\in W$, we denote their product
$w_1\star w_2\star\cdots\star w_n=\prod_{i=1}^n w_i$. In general, if
$F=\{n_1<n_2<\cdots<n_s\}$ is a finite ordered set of indexes, then
$\prod_{i\in F}w_i:=w_{n_1}\star w_{n_2}\star\cdots\star w_{n_s}$.

If $F,G\subset\N$ are finite nonempty sets, we write $F<G$ to mean $\max F<\min G$.

\begin{theorem}\label{RHHJ}
Let $W=W_A$ be the set of words over an alphabet $A$,
let $(B_n)_{n\in\N}$ be a sequence of finite subsets of $W$, and let $k\in\N$.
For every finite coloring of $[W]^k$ there exists a sequence of variable words
$(w_n(v))_{n\in\N}$ of increasing lengths such that
for every sequence of words $(b_n)_{n\in\N}$ where $b_n\in B_n$
and for all nonempty finite subsets $F_1<F_2<\cdots<F_k$ of $\N$,
the following $k$-tuples are monochromatic:
$$\left\{\prod_{i\in F_1}w_i(b_i),
\prod_{i\in F_2}w_i(b_i),\,\ldots\,,
\prod_{i\in F_k}w_i(b_i)\right\}.$$
\end{theorem}

Before proving the theorem, we observe that it actually ``directly" generalizes
strong Hales-Jewett, Ramsey, and Hindman's Theorems.

\begin{remark}\label{maintheoremimpliesHJ}
Given a finite set $B\subset W$,
apply the theorem above where $B_n:=B$ for all $n$ and where $k=1$.
By identifying $[W]^1$ with $W$, we obtain the following property:
\begin{itemize}
\item
\emph{For every finite coloring of $W=C_1\cup\ldots\cup C_r$ there exists a sequence of
variable words $(w_n(v))_{n\in\N}$ of increasing lengths such that for every sequence
$(b_n)_{n\in\N}$ in $B$ and for every nonempty finite subset $F_1\subset\N$,
all elements $\prod_{i\in F_1}w_i(b_i)\in C_j$ are monochromatic.}
\end{itemize}

Given $b\in B$, if we take a sequence $(b_n)_{n\in\N}$ where $b_1=b$
and we let $F_1=\{1\}$, we obtain that $w_1(b)\in C_j$.
So, the variable word $w(v):=w_1(v)$ is such that
the set of instances $\{w(b)\mid b\in B\}$ is monochromatic.
This shows that the strong version of Hales-Jewett Theorem
is a particular case of Theorem \ref{RHHJ}.
\end{remark}

\begin{remark}\label{maintheoremimpliesMT}
As already observed in Remark \ref{singleletter},
if the alphabet $A=\{a\}$ contains a single symbol, then
every word $w$ is uniquely determined by its length $\ell(w)$, and so
the set of words $W=W_A$ can be identified with $\N$
by means of the bijection $w\mapsto\ell(w)$.
In this case, we obtain this property:
\begin{itemize}
\item
\emph{For every finite coloring of $[\N]^k$ there exists a sequence
of variable words $(w_n(v))_{n\in\N}$ of increasing lengths $\ell_n:=\ell(w_n(v))$
such that for all nonempty finite subsets $F_1<\ldots<F_k$ of $\N$,
the following $k$-tuples are monochromatic:}
\begin{multline*}
\left\{\ell\!\left(\prod_{i\in F_1}w_i(a)\right),\,\ldots\,,\ell\!\left(\prod_{i\in F_k}w_i(a)\right)\right\}=
\\
=\left\{\sum_{i\in F_1}\ell(w_i(a))\,,\ldots\,,\sum_{i\in F_k}\ell(w_i(a))\right\}=
\left\{\sum_{i\in F_1}\ell_i,\,\ldots\,,\sum_{i\in F_k}\ell_i\right\}.
\end{multline*}
\end{itemize}

Note that this is precisely the property of Milliken-Taylor Theorem, where the
desired increasing sequence is $(\ell_n)_{n\in\N}$.
So, Milliken-Taylor Theorem, and hence both Ramsey and Hindman's Theorems,
are particular cases of Theorem \ref{RHHJ}.
\end{remark}

\begin{proof}[Proof of Theorem \ref{RHHJ}]
By Lemma \ref{lemma},
we can pick a variable hyperword $\w(v)\in{}^*W_{\{v\}}\setminus W_{\{v\}}$ such that
all instances $\w(b)$ for $b\in W$ are $u$-equivalent to each other
and idempotent, and so $\w(b)\ueq\w(b')\star\s(\w(b''))$ for all $b,b',b''\in W$.

Now let $[W]^k=C_1\cup\ldots\cup C_r$ be a finite coloring.
By the hypothesis on $\w(v)$, it follows that
the $k$-tuples $\{\w(b),\s(\w(b)),\ldots,\s^{k-1}(\w(b))\}$ for $b\in W$
are $u$-equivalent to each other, and so there exists a color $C=C_j$ such that
all $k$-tuples $\{\w(b),\s(\w(b)),\ldots,\s^{k-1}(\w(b))\}\in\s^k(C)$.
Fix any word $\overline{b}\in W$.
By recursion on $n$ we define variable words $w_n(v)$
of increasing lengths in such a way that
for all nonempty subsets $F_1<\ldots<F_h$ of $\{1,\ldots,n\}$
where $1\le h\le k$, and for all $b_1,\ldots,b_n$ where $b_i\in B_i$,
the following conditions are satisfied,
where for simplicity we denote $w_{F_s}=\prod_{i\in F_s}w_i(b_i)$:\footnote
{~When $h=k$, we convene that $(\dagger)_n$ means
$\{w_{\beta_1},\ldots,w_{\beta_k}\}\in\s^{0}(C)=C$; and when
$h=1$, that $(\ddagger)_n$ means
$\{w_{\beta_1}\star\,\w(\overline{b}),\s(\w(\overline{b})),\ldots,\s^{(k-1)}(\w(\overline{b}))\}\in\s^k(C)$.}
\begin{enumerate}
\item[$(\dagger)_n$]
$\{w_{F_1},\ldots,w_{F_h},\w(\overline{b}),\s(\w(\overline{b})),\ldots,\s^{(k-h-1)}(\w(\overline{b}))\}\in\s^{k-h}(C)$.
\item[$(\ddagger)_n$]
$\{w_{F_1},\ldots,w_{F_{h-1}},
w_{F_h}\star\,\w(\overline{b}),\s(\w(\overline{b})),\ldots,\s^{(k-h)}(\w(\overline{b}))\}\in\s^{k-h+1}(C)$.
\end{enumerate}

Clearly, properties $(\dagger)_n$ with $h=k$ guarantee that the
defined sequence \newline $(w_n(v))_{n\in\N}$
has the desired properties.

Let us start with the base case $n=1$.
Since all instances $\w(b)$ for $b\in W$ are $u$-equivalent to each other,
and since by idempotency $\w(b)\ueq\w(b)\star\s(\w(b))$,
we obtain the following $u$-equivalence for each $b_1\in B_1$:
$$(\w(b_1),\s(\w(\overline{b})),\ldots,\s^{k-1}(\w(\overline{b})))\ueq
(\w(b_1)\star\s(\w(\overline{b})),\s^2(\w(\overline{b})),\ldots,\s^k(\w(\overline{b}))),$$
and hence
$$\{\w(b_1),\s(\w(\overline{b})),\ldots,\s^{k-1}(\w(\overline{b}))\}\ueq
\{\w(b_1)\star\s(\w(\overline{b})),\s^2(\w(\overline{b})),\ldots,\s^k(\w(\overline{b}))\}.$$

So, we see that there exists a variable hyperword $x(v)\in{}^*W_{\{v\}}$, namely $\w(v)$,
that satisfies the following two properties for each $b_1\in B_1$:
\begin{itemize}
\item
$\{x(b_1),\s(\w(\overline{b})),\ldots,\s^{k-1}(\w(\overline{b}))\}\in\s^k(C)$.
\item
$\{x(b_1)\star\s(\w(\overline{b})),\s^2(\w(\overline{b})),\ldots,\s^k(\w(\overline{b}))\}\in\s^{k+1}(C)$.
\end{itemize}

Since $B_1$ is finite, we can apply backward \emph{transfer} to the finitely many properties above,
and obtain the existence of a variable word $w_1(v)\in W_{\{v\}}$
such that for every $b_1\in B_1$:
\begin{itemize}
\item
$\{w_1(b_1),\w(\overline{b}),\ldots,\s^{k-2}(\w(\overline{b}))\}\in\s^{k-1}(C)$.
\item
$\{w_1(b_1)\star\w(\overline{b}),\s(\w(\overline{b})),\ldots,\s^{k-1}(\w(\overline{b}))\}\in\s^k(C)$.
\end{itemize}

It is readily seen that the two properties above say that
$(\dagger)_1$ and $(\ddagger)_1$ are satisfied by $w_1(v)$.
Indeed, observe that if $F_1<\ldots<F_h$ are nonempty subsets of $\{1\}$, then
it must be $h=1$ and $F_1=\{1\}$.

Let's now move on to the inductive step, and assume that words
$w_1(v),\ldots,w_n(v)$ of increasing lengths have been defined that
satisfy $(\dagger)_n$ and $(\ddagger)_n$.
We observe that, since $\w(\overline{b})\in{}^*W$ is idempotent
and $w_{\beta_h}\in W$, one has the $u$-equivalences
$\w(\overline{b})\ueq \w(\overline{b})\star\,\s(\w(\overline{b}))$ and
$w_{\beta_h}\star\,\w(\overline{b})\ueq w_{\beta_h}\star\,\w(\overline{b})\star\,\s(\w(\overline{b}))$,
and so also
\begin{multline*}
\{w_{F_1},\ldots,w_{F_h},
\w(\overline{b}),\s(\w(\overline{b})),\ldots,\s^{(k-h-1)}(\w(\overline{b}))\}\ueq
\\
\{w_{F_1},\ldots,w_{F_h},
\w(\overline{b})\star\s(\w(\overline{b})),\s^2(\w(\overline{b})),\ldots,\s^{(k-h)}(\w(\overline{b}))\}
\end{multline*}
and
\begin{multline*}
\{w_{F_1},\ldots,w_{F_{h-1}},
w_{F_h}\star\,\w(\overline{b}),\s(\w(\overline{b})),\ldots,\s^{(k-h)}(\w(\overline{b}))\}\ueq
\\
\{w_{F_1},\ldots,w_{F_{h-1}},
w_{F_h}\star\,\w(\overline{b})\star\s(\w(\overline{b})),\s^2(\w(\overline{b})),\ldots,\s^{(k-h+1)}(\w(\overline{b}))\}
\end{multline*}
In consequence,
we can add to the inductive hypotheses $(\dagger)_m$ and $(\ddagger)_m$
the following two properties for each $b_{n+1}\in B_{n+1}$:\footnote
{~When $h=k$, we convene that $(\dagger')_m$ means
$\{w_{\beta_1},\ldots,w_{\beta_h}\}\in C$;
and when $h=1$, that $(\ddagger')_m$ means
$\{w_{\beta_1}\star\,\w(b_{n+1})\star\s(\w(\overline{b})),\s^2(\w(\overline{b})),\ldots,\s^k(\w(\overline{b}))\}\in\s^{k+1}(C)$.}

\begin{itemize}
\item[$(\dagger')_n$]
$\{w_{F_1},\ldots,w_{F_h},
\w(b_{n+1})\star\,\s(\w(\overline{b})),\s^2(\w(\overline{b})),\ldots,\s^{(k-h)}(\w(\overline{b}))\}\in\s^{k-h+1}(C)$.
\item[$(\ddagger')_n$]
$\{w_{F_1},\ldots,w_{F_{h-1}},
w_{F_h}\star\,\w(b_{n+1})\star\,\s(\w(\overline{b})),\s^2(\w(\overline{b})),\ldots,\s^{(k-h+1)}
(\w(\overline{b}))\}$

$\in\s^{k-h+2}(C)$.
\end{itemize}

Looking at properties $(\dagger)_m$, $(\dagger')_m$,
$(\ddagger)_m$, and $(\ddagger')_m$, we see
that there exists a variable hyperword $x(v)\in{}^*W_{\{v\}}$, namely $\w(v)$,
whose length $\ell(x(v))>\ell(w_m(x))$ (in fact $\w(v)\in{}^*W_{\{v\}}\setminus W_{\{v\}}$ has infinite length)
and that satisfies the following four conditions for all $b_1,\ldots,b_{n+1}$ where $b_i\in B_i$,
and for all nonempty subsets $F_1<\ldots<F_h$ of $\{1,\ldots,n\}$ where $1\le h\le k$:
\begin{itemize}
\item
$\{w_{F_1},\ldots,w_{F_h},x(b_{n+1}),\s(\w(\overline{b})),\ldots,\s^{(k-h-1)}(\w(\overline{b}))\}\in\s^{k-h}(C)$.
\item
$\{w_{F_1},\ldots,w_{F_h},x(b_{n+1})\star\s(\w(\overline{b})),
\s^2(\w(\overline{b})),\ldots,\s^{(k-h)}(\w(\overline{b}))\}\in\s^{k-h+1}(C)$.
\item
$\{w_{F_1},\ldots,w_{F_{h-1}},
w_{F_h}\star\,x(b_{n+1}),\s(\w(\overline{b})),\ldots,
\s^{(k-h)}(\w(\overline{b}))\}\in\s^{k-h+1}(C)$.
\item
$\{w_{F_1},\ldots,w_{F_{h-1}},
w_{F_h}\star\,x(b_{n+1})\star\,\s(\w(\overline{b})),
\s^2(\w(\overline{b})),\ldots,\s^{(k-h+1)}(\w(\overline{b}))\}$

$\in\s^{k-h+2}(C)$.
\end{itemize}

Since all sets $B_i$ are finite, we can apply by backward \emph{transfer} to the finitely many properties above,
and obtain the existence of a variable word $w_{m+1}(v)\in W_{\{v\}}$
with length $\ell(w_{m+1}(v))>\ell(w_m(v))$, and
such that for all $b_1,\ldots, b_{n+1}$ where $b_i\in B_i$, and for all nonempty subsets
$F_1<\ldots<F_h$ of $\{1,\ldots,n\}$ where $1\le h\le k$:
\begin{itemize}
\item[$(i)$]
$\{w_{F_1},\ldots,w_{F_h},w_{m+1}(b_{n+1}),\w(\overline{b}),\ldots,\s^{(k-h-2)}(\w(\overline{b}))\}\in\s^{k-h-1}(C)$.
\item[$(ii)$]
$\{w_{F_1},\ldots,w_{F_h},w_{m+1}(b_{n+1})\star\w(\overline{b}),\s(\w(\overline{b})),\ldots,\s^{(k-h-1)}(\w(\overline{b}))\}\in\s^{k-h}(C)$.
\item[$(iii)$]
$\{w_{F_1},\ldots,w_{F_{h-1}},
w_{F_h}\star\,w_{m+1}(b_{n+1}),\w(\overline{b}),\ldots,\s^{(k-h-1)}(\w(\overline{b}))\}\in\s^{k-h}(C)$.
\item[$(iv)$]
$\{w_{F_1},\ldots,w_{F_{h-1}}, w_{F_h}\star\,w_{m+1}(b_{n+1})\star\w(\overline{b}),
\s(\w(\overline{b})),\ldots,\s^{(k-h)}(\w(\overline{b}))\}$

$\in\s^{k-h+1}(C)$.
\end{itemize}

Now let $G_1<\ldots<G_h$ be nonempty subsets of $\{1,\ldots,n,n+1\}$ where $1\le h\le k$.
We have to show that for all $b_1,\ldots,b_{n+1}$ where $b_i\in B_i$,
the corresponding properties $(\dagger)_{n+1}$ and $(\ddagger)_{n+1}$ hold for
the sequence of variable words $w_1(v),\ldots,w_n(v),w_{n+1}(v)$.

If $n+1\notin G_h$, then we can directly apply the inductive hypotheses
$(\dagger)_n$ and $(\ddagger)_n$,
since in this case $G_1<\ldots<G_h$ are subsets of $\{1,\ldots,n\}$.

If $G_h=\{n+1\}$, then note that $w_{G_h}=w_{n+1}(b_{n+1})$.
So, $(\dagger)_{n+1}$ and $(\ddagger)_{n+1}$ follow from $(i)$ and $(ii)$ respectively,
by considering the subsets $G_1<\ldots<G_{h-1}$ of $\{1,\ldots,n\}$.

If $n+1\in G_h$ and $G_h$ is not a singleton, then
note that $w_{G_h}=w_{G'_h}\star w_{n+1}(b_{n+1})$
where $G'_h:=G_h\setminus\{n+1\}\ne\emptyset$.
In this case $(\dagger)_{n+1}$ and $(\ddagger)_{n+1}$ follow from $(iii)$
and $(iv)$ respectively, by considering the nonempty subsets
$G_1<\ldots<G_{h-1}<G'_h$ of $\{1,\ldots,n\}$.
\end{proof}

\bigskip

\bibliographystyle{amsalpha}

\end{document}